\documentclass[11pt]{article}
\usepackage{amscd}
\usepackage{amsfonts}
\usepackage{amsmath}
\usepackage{amssymb}
\usepackage{amsthm}
\usepackage{bbm}
\usepackage{CJK}
\usepackage{fancyhdr}
\usepackage{graphicx}
\usepackage{indentfirst}
\usepackage{latexsym,bm}
\usepackage{mathrsfs}
\usepackage[square, comma, sort&compress, numbers]{natbib}
\usepackage{xcolor}
\usepackage{soul}
\allowdisplaybreaks[4]
\newtheorem{theorem}{Theorem}[section]
\newtheorem{lemma}[theorem]{Lemma}
\newtheorem{definition}[theorem]{Definition}
\newtheorem{proposition}[theorem]{Proposition}
\newtheorem{example}[theorem]{Example}

\newtheorem{corollary}[theorem]{Corollary}
\newtheorem{remark}[theorem]{Remark}

\usepackage[top=1in,bottom=1in,left=1.25in,right=1.25in]{geometry}
\def\di{\diamond}

\def\ot{\otimes}

\date{}
\begin{document}
\renewcommand{\baselinestretch}{1.2}
\renewcommand{\arraystretch}{1.0}
\title{\bf Fundamental theorem of Poisson 3-Lie $(A,H)$-Hopf modules}
 \date{}
\author {{\bf Yan Ning$^1$, Daowei Lu$^1$\footnote {Corresponding author:  ludaowei620@126.com}, Dingguo Wang$^{2,3}$}\\
{\small $^1$School of Mathematics and Big Data, Jining University}\\
{\small Qufu, Shandong 273155, P. R. China}\\
{\small $^2$Department of General Education, Shandong Xiehe University}\\
{\small Jinan, Shandong, 250109, P. R. China}\\
{\small $^3$School of Mathematical Sciences, Qufu Normal University}\\
{\small Qufu, Shandong 273165, P. R. China}
}
 \maketitle
\begin{center}
\begin{minipage}{12.cm}

\noindent{\bf Abstract.} Let $H$ be a Hopf algebra with a bijective antipode and $A$ an $H$-comodule Poisson 3-Lie algebra. Assume that there exists an $H$-colinear map which is also an algebra map from $H$ to the Poisson center of $A$. In this paper we generalize the fundamental theorem of $(A, H)$-Hopf modules to Poisson 3-Lie $(A, H)$-Hopf modules and deduce relative projectivity in the category of Poisson 3-Lie $(A, H)$-Hopf modules.
\\

\noindent{\bf Keywords:} Poisson 3-Lie algebra, Hopf algebra, Poisson 3-Lie Hopf module,  Fundamental theorem.
\\

 \noindent{\bf  Mathematics Subject Classification:} 17B63, 16T05.
 \end{minipage}
 \end{center}
 \normalsize\vskip1cm

\section*{Introduction}

Hopf algebras provide an algebraic framework for describing symmetries and have played an important role in the study of modules, comodules, and various algebraic structures with symmetry. One of the fundamental results in Hopf algebra theory is the fundamental theorem of Hopf modules, which describes Hopf modules in terms of their coinvariant elements. This result has motivated a number of extensions to module categories associated with additional algebraic structures.

The theory of $n$-Lie algebras, introduced by Filippov \cite{Filippov}, is a natural generalization of the theory of Lie algebras. In particular, 3-Lie algebras have received considerable attention in algebra and mathematical physics. They are closely related to generalized Hamiltonian mechanics and have also appeared in the study of multiple $M2$-branes; see \cite{Na,Tak,Az,BL1,BL2,HHM,FT}. The essential feature of a 3-Lie algebra is a skew-symmetric ternary operation satisfying the fundamental identity. This identity guarantees that the adjoint action of two elements is a derivation and provides the basic algebraic mechanism underlying the theory of 3-Lie algebras.

Poisson algebras combine a commutative associative algebra structure with a compatible Lie algebra structure. In the same spirit, Poisson 3-Lie algebras combine a commutative associative multiplication with a compatible 3-Lie algebra structure. This interaction gives rise to corresponding module theories in which both the associative action and the ternary action have to satisfy suitable compatibility conditions. Therefore, when Hopf algebra coactions are incorporated into the setting of Poisson 3-Lie algebras, the resulting module theory involves the simultaneous compatibility of three different structures: the associative multiplication, the 3-Lie bracket, and the Hopf algebra coaction.

The fundamental theorem for Poisson Hopf modules was established by Guédénon \cite{Gu}. More precisely, the Poisson $(A,H)$-Hopf module setting provides a framework in which the classical Hopf module theorem can be combined with the Poisson structure, and the corresponding Hopf modules can be described in terms of their coinvariant parts. This direction has subsequently been extended to other Hopf-algebraic settings. In particular, Lu and Wang \cite{LW1} established a fundamental theorem of Poisson Hopf modules for weak Hopf algebras. Related results for Hopf group coalgebras and braided Hopf algebras were obtained in \cite{LW2,LW3}. These developments suggest that fundamental theorems for Poisson Hopf modules can be formulated in increasingly general algebraic settings.

It is therefore natural to ask whether an analogous fundamental theorem can be established for Poisson 3-Lie Hopf modules. The presence of the ternary operation makes this question nontrivial. In contrast to the ordinary Poisson case, a 3-Lie module involves an action of pairs of elements, and this action must satisfy the identities associated with the fundamental identity of the 3-Lie algebra. Moreover, the ternary action has to be compatible with the associative $A$-module structure and the $H$-comodule structure. Thus, the methods for ordinary Poisson Hopf modules need to be adapted to take the additional 3-Lie structure into account.

In this paper, we introduce and study Poisson 3-Lie $(A,H)$-Hopf modules. Let $H$ be a Hopf algebra with bijective antipode and let $A$ be a right $H$-comodule Poisson 3-Lie algebra. We first establish the category of Poisson 3-Lie $(A,H)$-Hopf modules and investigate its basic properties. In particular, we study the relationship between the coinvariant subspaces and the corresponding 3-Lie module structures. Under suitable assumptions, we obtain injectivity results for Poisson 3-Lie $(A,H)$-Hopf modules, which are analogous to the corresponding properties in the Poisson Hopf module setting.

Our main result is a fundamental theorem for Poisson 3-Lie $(A,H)$-Hopf modules. Suppose that there exists a right $H$-colinear algebra map
$$\phi:H\rightarrow A^A,$$
where $A^A$ denotes the center of the 3-Lie algebra $A$. For a Poisson 3-Lie $(A,H)$-Hopf module $M$, we consider its coinvariant subspace and construct the natural map
$$\alpha:A\otimes_B M^{AcoH}\rightarrow M, \qquad a\otimes_Bm\mapsto a\cdot m,$$
where $B=A^{AcoH}$, the intersection of $A^A$ and $A^{coH}$. 

When the induced 3-Lie actions on $M^{coH}$ and $A^{coH}$ are trivial, we prove that $\alpha$ is an isomorphism of Poisson 3-Lie $(A,H)$-Hopf modules. The inverse map is obtained explicitly by using the map $\phi$ together with the canonical projection onto the coinvariant subspace. Hence, under these assumptions, a Poisson 3-Lie $(A,H)$-Hopf module is completely determined by its appropriate coinvariant part.

We also consider Poisson $H$-simple $H$-comodule Poisson 3-Lie algebras. In this situation, we show that the algebra of appropriate coinvariants has strong algebraic properties. These results provide applications of the fundamental theorem and further illustrate the role of coinvariant elements in the structure theory of Poisson 3-Lie $(A,H)$-Hopf modules.

The paper is organized as follows. In Section 1, we recall the basic definitions and properties of 3-Lie algebras, Poisson 3-Lie algebras, and their modules. In Section 2, we introduce the category of Poisson 3-Lie $(A,H)$-Hopf modules and establish several preliminary and homological results. In Section 3, we prove the fundamental theorem of Poisson 3-Lie $(A,H)$-Hopf modules and derive its consequences.

\vspace{2mm}

\noindent

\section{Preliminaries}
\def\theequation{1.\arabic{equation}}
\setcounter{equation} {0}

In this paper, let $k$ be a fixed field of characteristic 0. All the objects under discussion, including vector spaces, algebras and tensor products, are taken over $k$ unless otherwise specified.
For any coalgebra $(C,\Delta,\varepsilon)$, we compress the Sweedler notation of the comultiplication $\Delta$ as $\Delta(c)=c_1\otimes c_2$ for simplicity.

\begin{definition}\cite{Filippov}
  A 3-Lie algebra is a vector space $A$ together with a skew-symmetric linear map ($3$-Lie bracket) $\{-,-,-\}:
\otimes^3 A\rightarrow A$ such that the following {\bf Fundamental Identity} holds:
\begin{equation}
\{\{x,y,z\},u,v\}=\{\{x,u,v\},y,z\}+\{\{y,u,v\},z,x\}+\{\{z,u,v\},x,y\},\label{1a}
\end{equation}
for $x, y, z, u, v\in A$.
\end{definition}
In other words, the adjoint action of $u,v$ in the ternary operation is a derivation.

A subspace $B$ of $A$ is called a subalgebra of $A$ if $\{B,B,B\}\subseteq B$.

The center $A^A$ of  3-Lie algebra $A$ is Lie-$A$-invariant elements of $A$, that is, 
$$A^A=\left\{b\in A|\{b,A,A\}=0\right\}.$$

\begin{definition}\cite{Ma}
Let $A$ be a 3-Lie algebra. A 3-Lie $A$-module is a pair $(M,\di)$ where $M$ is a vector space and $\di:L\times L\times M\rightarrow M,\ (x, y)\ot m\mapsto(x,y)\di m$ is a linear map on the three items satisfying 
\begin{align}
&(x, y)\di m=-(y, x)\di m,\label{1b}\\
&(x, y)\di((u,v)\di m)-(u,v)\di((x, y)\di m)=(\{x,y,u\},v)\di m+(u,\{x,y,v\})\di m,\label{1c}\\
&(\{x,y,u\},v)\di m=(y,u)\di((x,v)\di m)-(x,u)\di((y,v)\di m)+(x,y)\di((u,v)\di m),\label{1d}
\end{align}
for all $x,y,u,v\in A,m\in M$.
\end{definition}

Let $M,N$ be two  3-Lie $A$-modules. A linear map $f:M\rightarrow N$ is called a morphism of 3-Lie $A$-modules if for all $x,y\in A,m\in M$,
$$f((x,y)\di m)=(x,y)\di f(m).$$

For 3-Lie $A$-module $M$, the Lie $A$-invariant elements of $M$ is
$$M^A=\left\{m\in M|(x,y)\di m=0,\forall x,y\in A\right\}.$$

\begin{definition}\cite{Dz}
 A Poisson 3-Lie algebra is a vector space $A$ with a bilinear operator $\cdot$ and a ternary operator $\{-,-,-\}$ such that $(A,\cdot)$ is a commutative associative algebra, $(A, \{-,-,-\})$ is a 3-Lie algebra and the following equation holds:
\begin{equation}
\{x,y,uv\}=u \{x,y,v\}+\{x,y,u\} v,\label{1e}
\end{equation}
for all $x,y,u,v\in A.$
\end{definition}
Thus the adjoint action of the pair $(x,y)$ on the commutative associative product is a derivation. Let $u=v=1$, we have $\{x,y,1\}=0$.

Let $A$ and $T$ be two Poisson 3-Lie algebras. A homomorphism of Poisson 3-Lie algebra is an algebra map $f:A\rightarrow T$ such that
$$f(\{x,y,z\})=\{f(x),f(y),f(z)\},$$
for all $x,y,z\in A$.

A subspace $B$ of $A$ is called a Poisson 3-Lie subalgebra of $A$ if $B$ is both a subalgebra and 3-Lie subalgebra of $A$.

\begin{definition}
Let $A$ be a Poisson 3-Lie algebra. A vector space $M$ is called a Poisson 3-Lie $A$-module if $(M,\cdot)$ is an $A$-module and $(M,\di)$ is a 3-Lie $A$-module satisfying the compatible conditions
\begin{align}
&(xy,z)\di m=x\cdot((y,z)\di m)+y\cdot((x,z)\di m),\label{1f}\\
&(x,y)\di(z\cdot m)=\{x,y,z\}\cdot m+z\cdot((x,y)\di m),\label{1g}
\end{align}
for all $x,y,z\in A,m\in M$.
\end{definition}
A morphism of Poisson 3-Lie $A$-modules is both an $A$-linear and a 3-Lie $A$-linear map.

In the identity (\ref{1f}), let $x=y=1$, we obtain that $(1,z)\di m=0$ for all $z\in A,m\in M$.

\section{The category of Poisson 3-Lie $(A,H)$-Hopf modules}
\def\theequation{2.\arabic{equation}}
\setcounter{equation} {0}

\begin{definition}
Let $H$ be a Hopf algebra and $A$ a Poisson 3-Lie algebra. Then $A$ is called a right $H$-comodule Poisson 3-Lie algebra if $A$ is a right $H$-comodule algebra and satisfies the following relation
\begin{equation}
\{x,y,z\}_{(0)}\ot\{x,y,z\}_{(1)}=\{x_{(0)},y_{(0)},z_{(0)}\}\ot x_{(1)}y_{(1)}z_{(1)},\label{2a}
\end{equation}
for all $x,y,z\in A$.
\end{definition}

\begin{remark}
Since the 3-Lie bracket is skew-symmetry, we have
\begin{align*}
\{x,y,z\}_{(0)}\ot\{x,y,z\}_{(1)}&=-\{x,z,y\}_{(0)}\ot\{x,z,y\}_{(1)}\\
&=-\{x_{(0)},z_{(0)},y_{(0)}\}\ot x_{(1)}z_{(1)}y_{(1)}\\
&=\{x_{(0)},y_{(0)},z_{(0)}\}\ot x_{(1)}z_{(1)}y_{(1)},
\end{align*}
thus
$$\{x_{(0)},y_{(0)},z_{(0)}\}\ot x_{(1)}y_{(1)}z_{(1)}=\{x_{(0)},y_{(0)},z_{(0)}\}\ot x_{(1)}z_{(1)}y_{(1)}.$$
And similarly
$$\{x_{(0)},y_{(0)},z_{(0)}\}\ot x_{(1)}y_{(1)}z_{(1)}=\{x_{(0)},y_{(0)},z_{(0)}\}\ot y_{(1)}x_{(1)}z_{(1)}.$$
\end{remark}

\begin{lemma}\label{lem:2a}
Let $H$ be a Hopf algebra and $A$ a Poisson 3-Lie algebra. Then $A^{coH}$ is an $H$-subcomodule Poisson 3-Lie algebra of $A$.
\end{lemma}

\begin{proof}
It is well known that $A^{coH}$ is an $H$-subcomodule algebra of $A$. For all $x,y,z\in A^{coH}$,
$$\{x,y,z\}_{(0)}\ot\{x,y,z\}_{(1)}=\{x_{(0)},y_{(0)},z_{(0)}\}\ot x_{(1)}y_{(1)}z_{(1)}=\{x,y,z\}\ot 1_H,$$
hence $\{x,y,z\}\in A^{coH}$, that is, $A^{coH}$ is a Lie subalgebra of $A$.
\end{proof}

\begin{definition}
Let $A$ be a right $H$-comodule Poisson 3-Lie algebra. A vector space $M$ is called a Poisson 3-Lie $(A,H)$-Hopf module if $M$ is a Poisson 3-Lie $A$-module and an $(A,H)$-Hopf module such that for all $x,y\in A,m\in M$,
\begin{equation}
((x,y)\di m)_{(0)}\ot ((x,y)\di m)_{(1)}=(x_{(0)},y_{(0)})\di m_{(0)}\ot x_{(1)}y_{(1)}m_{(1)}.\label{2b}
\end{equation}
\end{definition}

\begin{remark}
It is clear that an $H$-comodule Poisson 3-Lie algebra $A$ itself is Poisson 3-Lie $(A,H)$-Hopf module.
\end{remark}

\begin{lemma}\label{lem:2b}
Let $H$ be a Hopf algebra, $A$ a right $H$-comodule Poisson 3-Lie algebra, and $M$ a Poisson 3-Lie $(A,H)$-Hopf module. Then
\begin{itemize}
  \item [(1)] $M^A$ is an $H$-subcomodule of $M$.
  \item [(2)] $A^A$ is an $H$-subcomodule Poisson 3-Lie algebra of $A$, and $A^{AcoH}$ is a Poisson 3-Lie subalgebra of $A^A$.
  \item [(3)] $M^{AcoH}$ is a Poisson 3-Lie $A^{AcoH}$-submodule of $M$.
\end{itemize}
\end{lemma}

\begin{proof}
(1) For all $x,y\in A,m\in M^A$, since
\begin{align*}
&(1_A\ot S^{-1}((x,y)_{(1)}))\cdot[((x,y)_{(0)}\di m)_{(0)}\ot((x,y)_{(0)}\di m)_{(1)}]\\
&=(1_A\ot S^{-1}((x,y)_{(1)2}))\cdot[(x,y)_{(0)}\di m_{(0)}\ot (x,y)_{(1)1}m_{(1)}]\\
&=(1_A\ot S^{-1}(x_{(1)2}y_{(1)2}))\cdot[(x_{(0)},y_{(0)})\di m_{(0)}\ot x_{(1)1}y_{(1)1}m_{(1)}]\\
&=(x,y)\di m_{(0)}\ot m_{(1)}.
\end{align*}
By a similar argument as \cite{Gu}, we conclude that $M^A$ is an $H$-subcomodule of $M$.

(2) By (1) $A^A$ is an $H$-subcomodule of $A$. Since $A^A$ is also a subalgebra by (\ref{1e}) and 3-Lie subalgebra of $A$, $A^A$ is an $H$-subcomodule Poisson 3-Lie algebra of $A$. Similar argument implies that $A^{AcoH}$ is a Poisson 3-Lie subalgebra of $A^A$.

(3) Firstly $M$ is a Poisson 3-Lie $A^{AcoH}$-module of $M$ since $A^{AcoH}$ is a Poisson 3-Lie subalgebra of $A$. Then for all $x,y\in A,a\in A^{AcoH},m\in M^{AcoH}$,
$$(x,y)\di(a\cdot m)=\{x,u,a\}\cdot m+a\cdot ((x,y)\di m)=0,$$
that is, $a\cdot m\in M^{AcoH}$ and $M^{AcoH}$ is an $A^{AcoH}$-module. It is obvious that $M^{AcoH}$ is a 3-Lie $A^{AcoH}$-module. 
\end{proof}

In what follows, we will say that a vector space $M$ is an $(\underline{A},H)$-comodule if $M$ is a 3-Lie $A$-module, a right $H$-comodule and the relation (\ref{2b}) is satisfied. We denote by $_{\underline{A}}\mathcal{M}$ the category of 3-Lie $A$-modules with Lie $A$-linear maps, and by $_{\underline{A}}\mathcal{M}^H$ the category of $(\underline{A},H)$-comodule with 3-Lie $A$-linear and $H$-colinear maps. For $M,N\in\! _{\underline{A}}\mathcal{M}^H$, we denote by $_{\underline{A}}\textup{Hom}^H(M,N)$ the vector space of 3-Lie $A$-linear and $H$-colinear maps from $M$ to $N$.

\begin{lemma}\label{lem:2c}
\begin{itemize}
  \item [(1)] Let $N$ be a 3-Lie $A$-module. Then $N\ot H$ is an $(\underline{A},H)$-comodule with the 3-Lie $A$-action and $H$-coaction given by
\begin{align*}
&(x,y)\di (n\ot h)=(x_{(0)},y_{(0)})\di n\ot x_{(1)}y_{(1)}h,\\
&(n\ot h)_{(0)}\ot(n\ot h)_{(1)}=n\ot h_1\ot h_2,
\end{align*}
for all $x,y\in A, n\ot h\in N\ot H$.
  \item [(2)] Furthermore, if $H$ is commutative and $N$ is a Poisson 3-Lie $A$-module, then $N\ot H$ is a Poisson 3-Lie $(A, H)$-Hopf module with the $A$-action given by
$$x\cdot (n\ot h)= x_{(0)}\cdot n\ot x_{(1)}h,$$
for all $x,y\in A, n\ot h\in N\ot H$.
\end{itemize}
\end{lemma}

\begin{proof}
(1) For all $x,y,u,v\in A,n\ot h\in N\ot H$, first of all,
\begin{align*}
&(x,y)\di((u,v)\di (n\ot h))-(u,v)\di((x,y)\di (n\ot h))\\
&=(x,y)\di((u_{(0)},v_{(0)})\di n\ot u_{(1)}v_{(1)}h))-(u,v)\di(x_{(0)},y_{(0)})\di n\ot x_{(1)}y_{(1)}h)\\
&=(x_{(0)},y_{(0)})\di((u_{(0)},v_{(0)})\di n)\ot x_{(1)}y_{(1)}u_{(1)}v_{(1)}h\\
&\quad -(u_{(0)},v_{(0)})\di((x_{(0)},y_{(0)})\di n)\ot u_{(1)}v_{(1)}x_{(1)}y_{(1)}h\\
&=[(\{x_{(0)},y_{(0)},u_{(0)}\},v_{(0)})\di n+(u_{(0)},\{x_{(0)},y_{(0)},v_{(0)}\})\di n]\ot u_{(1)}x_{(1)}y_{(1)}v_{(1)}h\\
&=(\{x_{(0)},y_{(0)},u_{(0)}\},v_{(0)})\di n\ot x_{(1)}y_{(1)}u_{(1)}v_{(1)}h+(u_{(0)},\{x_{(0)},y_{(0)},v_{(0)}\})\di n\ot u_{(1)}x_{(1)}y_{(1)}v_{(1)}h\\
&=(\{x,y,u\},v)\di(n\ot h)+(u,\{x,y,v\})\di(n\ot h),
\end{align*}
then
\begin{align*}
&(\{x,y,u\},v)\di(n\ot h)+(x,u)\di((y,v)\di (n\ot h))\\
&=(\{x,y,u\}_{(0)},v_{(0)})\di n\ot \{x,y,u\}_{(1)}v_{(1)}h+(x,u)\di((y_{(0)},v_{(0)})\di n\ot y_{(1)}v_{(1)}h)\\
&=(\{x_{(0)},y_{(0)},u_{(0)}\},v_{(0)})\di n\ot x_{(1)}y_{(1)}u_{(1)}v_{(1)}h+(x_{(0)},u_{(0)})\di((y_{(0)},v_{(0)})\di n)\ot x_{(1)}u_{(1)}y_{(1)}v_{(1)}h\\
&=[(\{x_{(0)},y_{(0)},u_{(0)}\},v_{(0)})\di n+(x_{(0)},u_{(0)})\di((y_{(0)},v_{(0)})\di n)]\ot x_{(1)}y_{(1)}u_{(1)}v_{(1)}h\\
&=[(y_{(0)},u_{(0)})\di((x_{(0)},v_{(0)})\di n)+(x_{(0)},y_{(0)})\di ((u_{(0)},v_{(0)})\di n)]\ot x_{(1)}y_{(1)}u_{(1)}v_{(1)}h\\
&=(y_{(0)},u_{(0)})\di((x_{(0)},v_{(0)})\di n)\ot y_{(1)}u_{(1)}x_{(1)}v_{(1)}h+(x_{(0)},y_{(0)})\di ((u_{(0)},v_{(0)})\di n)\ot x_{(1)}y_{(1)}u_{(1)}v_{(1)}h\\
&=(y,u)\di((x,v)\di(n\ot h))+(x,y)\di((u,v)\di(n\ot h)).
\end{align*}
And
\begin{align*}
&((x,y)\di (n\ot h))_{(0)}\ot((x,y)\di (n\ot h))_{(1)}\\
&=((x_{(0)},y_{(0)})\di n\ot x_{(1)}y_{(1)}h)_{(0)}\ot((x_{(0)},y_{(0)})\di n\ot x_{(1)}y_{(1)}h)_{(1)}\\
&=(x_{(0)},y_{(0)})\di n\ot (x_{(1)}y_{(1)}h)_{1}\ot(x_{(1)}y_{(1)}h)_{2}\\
&=(x_{(0)},y_{(0)})\di n\ot x_{(1)1}y_{(1)1}h_1\ot x_{(1)2}y_{(1)2}h_2\\
&=(x_{(0)(0)},y_{(0)(0)})\di n\ot x_{(0)(1)}y_{(0)(1)}h_1\ot x_{(1)}y_{(1)}h_2\\
&=(x,y)_{(0)}\di (n\ot h)_{(0)}\ot(x,y)_{(1)} (n\ot h)_{(1)}.
\end{align*}
Thus $N\ot H$ is an $(\underline{A},H)$-comodule.

(2) For all  $x,y,z\in A,n\ot h\in N\ot H$,
\begin{align*}
(xy,z)\di (n\ot h)&=(x_{(0)}y_{(0)},z_{(0)})n\ot x_{(1)}y_{(1)}z_{(1)}h\\
&=x_{(0)}\cdot((y_{(0)},z_{(0)})\di n)\ot x_{(1)}y_{(1)}z_{(1)}h+y_{(0)}\cdot((x_{(0)},z_{(0)})\di n)\ot x_{(1)}y_{(1)}z_{(1)}h\\
&=x\cdot((y_{(0)},z_{(0)})\di n\ot y_{(1)}z_{(1)}h)+y\cdot((x_{(0)},z_{(0)})\di n\ot x_{(1)}z_{(1)}h)\\
&=x\cdot((y,z)\di(n\ot h))+y\cdot((x,z)\di(n\ot h)),
\end{align*}
and we have
\begin{align*}
(x,y)\di(z\cdot (n\ot h))&=(x,y)\di(z_{(0)}\cdot n\ot z_{(1)}h)\\
&=(x_{(0)},y_{(0)})\di (z_{(0)}\cdot n)\ot x_{(1)}y_{(1)}z _{(1)}h\\
&=\{x_{(0)},y_{(0)},z_{(0)}\}\cdot n\ot x_{(1)}y_{(1)}z _{(1)}h+z_{(0)}\cdot((x_{(0)},y_{(0)})\di n)\ot x_{(1)}y_{(1)}z _{(1)}h\\
&=\{x,y,z\}\cdot(n\ot h)+z\cdot((x,y)\di(n\ot h)),
\end{align*}
where we used the commutativity of $H$. By (1), the proof is completed.
\end{proof}

\begin{lemma}\label{lem:2d}
\begin{itemize}
  \item [(1)] Let $M$ be an $(\underline{A},H)$-comodule and $N$ a Lie $A$-module. There exists a linear isomorphism
$$\gamma:\! _{\underline{A}}\textup{Hom}^H(M,N\ot H)\rightarrow \! _{\underline{A}}\textup{Hom}(M,N),\ f\mapsto(id\ot\varepsilon)\circ f.$$
The inverse $\gamma'$ of $\gamma$ is given by $\gamma'(g)=(g\ot id_H)\circ\rho$.
  \item [(2)] Let $H$ be commutative, $M$ a Poisson 3-Lie $(A, H)$-Hopf module and $N$ a Poisson 3-Lie $A$-module. There exists a linear isomorphism
$$\gamma:\! _{\mathcal{P}A}\textup{Hom}^H(M,N\ot H)\rightarrow\! _{\mathcal{P}A}\textup{Hom}(M,N),\ f\mapsto(id\ot\varepsilon)\circ f.$$
\end{itemize}
\end{lemma}

\begin{proof}
(1) By Lemma \ref{lem:2c}, $N \otimes H$ is an $(\underline{A}, H)$-comodule. Let $f \in\! _{\mathcal{P}A}Hom^H(M, N \otimes H)$ and $m \in M$. Denote $ f(m) = \sum_{i} (n_i \otimes h_i)$ for some $n_i \in N$ and $h_i \in H$. Then $\gamma(f)(m) = \sum_{i \in I} (n_i \varepsilon(h_i))$. For $x,y\in A$, $m \in M$, we have
$$f((x,y)\di m)=(x,y)\diamond f(m)=\sum_{i} (x_{(0)},y_{(0)}) \di n_i\otimes x_{(1)}y_{(1)}h_i,$$
so
\begin{align*}
\gamma(f)((x,y)\di m)&=(id\ot\varepsilon)f((x,y)\di m)\\
&=(x_{(0)},y_{(0)}) \di n_i\varepsilon(x_{(1)}y_{(1)}h_i)\\
&=(x,y) \di n_i\varepsilon(h_i)\\
&=(x,y) \di \gamma(f)(m),
\end{align*}
thus $\gamma(f)$ is Lie $A$-linear. Now for $g\in\! _{\underline{A}}\textup{Hom}(M,N).$ It is clear that $\gamma'(g)$ is a morphism of $H$-comodule. For all $x,y\in A,m\in M$, we have
\begin{align*}
\gamma'(g)((x,y)\di m)&=g(((x,y)\di m)_{(0)})\ot((x,y)\di m)_{(1)}\\
&=g((x_{(0)},y_{(0)})\di m_{(0)})\ot x_{(1)}y_{(1)} m_{(1)}\\
&=(x_{(0)},y_{(0)})\di g( m_{(0)})\ot  x_{(1)}y_{(1)}  m_{(1)}\\
&=(x,y)\di (g( m_{(0)})\ot m_{(1)})\\
&=(x,y)\di (\gamma'(g)(m)),
\end{align*}
thus $\gamma'(g)$ is Lie $A$-linear. 
By the same computations as \cite{Gu}, we obtain that $\gamma'$ is the inverse of $\gamma$.

(2) Since $N\ot H$ is a Poisson 3-Lie $(A,H)$-Hopf module, and by a straightforward computation we get that $\gamma(f)$ and $\gamma'(g)$ are $A$-linear. The conclusion follows.
\end{proof}

The following result is direct by Lemma \ref{lem:2d}.
\begin{corollary}\label{cor:2c}
Let $H$ be commutative. If $N$ is an injective Poisson 3-Lie $A$-module, then $N\ot H$ is an injective Poisson 3-Lie $(A, H)$-Hopf module.
\end{corollary}

\begin{theorem}\label{thm:2d}
Let $A$ be a $H$-comodule Poisson 3-Lie algebra, and $M$ a Poisson 3-Lie $(A, H)$-Hopf module. Assume that there exists an $H$-colinear map $\phi:H\rightarrow A^A$ such that $\phi(1_H)=1_A$.
\begin{itemize}
  \item [(1)]  If $H$ is commutative or $\phi$ is an algebra map, then every Poisson 3-Lie $(A, H)$-Hopf module which is injective as a Lie $A$-module is an injective $(\underline{A}, H)$-comodule.
  \item [(2)] If $H$ is commutative, then every Poisson 3-Lie $(A, H)$-Hopf module which is injective Poisson 3-Lie $A$-module is an injective Poisson 3-Lie $(A, H)$-Hopf module.
\end{itemize}
\end{theorem}

\begin{proof}
(1) Let $M$ be a Poisson 3-Lie $(A, H)$-Hopf module and consider the linear map $\lambda:M\ot H\rightarrow M$ defined by
$$\lambda(m\ot h)=\phi(hS^{-1}(m_{(1)}))\cdot m_{(0)},\quad \forall m\in M,h\in H.$$
Just as shown in \cite{Gu}, $\lambda\circ\rho_M=id_M$ and $\lambda$ is $H$-colinear, which means that $\rho_M$ is an injective map. 

For all $x,y\in A,m\in M,h\in H$, we have
\begin{align*}
\lambda((x,y)\di (m\ot h))&=\lambda((x_{(0)},y_{(0)})\di m\ot x_{(1)}y_{(1)}h)\\
&=\phi(x_{(1)}y_{(1)}hS^{-1}(((x_{(0)},y_{(0)})\di m)_{(1)}))\cdot ((x_{(0)},y_{(0)})\di m)_{(0)}\\
&=\phi(x_{(1)}y_{(1)}hS^{-1}(x_{(0)(1)}y_{(0)(1)} m_{(1)}))\cdot ((x_{(0)(0)},y_{(0)(0)})\di m_{(0)})\\
&=\phi(x_{(1)2}y_{(1)2}hS^{-1}(x_{(1)1}y_{(1)1} m_{(1)}))\cdot ((x_{(0)},y_{(0)})\di m_{(0)})\\
\end{align*}
and
\begin{align*}
(x,y)\di\lambda(m\ot h)&=(x,y)\di(\phi(hS^{-1}(m_{(1)}))\cdot m_{(0)})\\
&=\{x,y,\phi(hS^{-1}(m_{(1)}))\}\cdot m_{(0)}+\phi(hS^{-1}(m_{(1)}))\cdot((x,y)\di m_{(0)})\\
&=\phi(hS^{-1}(m_{(1)}))\cdot((x,y)\di m_{(0)}),
\end{align*}
where the last identity holds since $\phi(hS^{-1}_H(m_{(1)}))\in A^A$. Now if $H$ is commutative, we obtain
$$\lambda((x,y)\di (m\ot h))=(x,y)\di\lambda(m\ot h),$$
that is, $\lambda$ is Lie $A$-linear and a homomorphism in $_{\underline{A}}\mathcal{M}^H$. If $\phi$ is an algebra map, using the fact that $A$ is commutative, we also have $\lambda((x,y)\di (m\ot h))=(x,y)\di\lambda(m\ot h).$ Thus $\lambda$ is Lie $A$-linear and a homomorphism in $_{\underline{A}}\mathcal{M}^H$. For all $x,y\in A,m\in M$,
\begin{align*}
\rho_M((x,y)\di m)&=((x,y)\di m)_{(0)}\ot((x,y)\di m)_{(1)}\\
&=(x_{(0)},y_{(0)})\di m_{(0)}\ot x_{(1)}y_{(1)} m_{(1)}\\
&=(x,y)\di \rho_M(m),
\end{align*}
hence $\rho_M$ is a 3-Lie $A$-linear.
 It is easy to verify that $\rho_M$ is $H$-colinear, so it is a homomorphism of $(\underline{A},H)$-comodule. If $M$ is an injective 3-Lie $A$-module, by a similar argument as in \cite{Gu}, $M$ is an injective $(\underline{A},H)$-comodule.

(2) By a similar computation as above, we obtain that $\lambda$ given in (1) is $A$-linear. Thus $\lambda$ is a homomorphism of Poisson 3-Lie $(A, H)$-Hopf modules. By the same arguments as \cite{Gu}, we conclude that an injective Poisson 3-Lie $A$-module is also an injective Poisson 3-Lie $(A,H)$-Hopf module.
\end{proof}

\begin{corollary}
Let $A$ be a commutative $H$-comodule algebra, and $M$ an $(A, H)$-Hopf module. Assume that there is an $H$-colinear map $\phi:H\rightarrow A$ such that $\phi(1_H)=1_A$.
\begin{itemize}
  \item [(1)] Then every $(A, H)$-Hopf module is an injective $H$-comodule.
  \item [(2)] Let $H$ be commutative. Then every $(A, H)$-Hopf module which is injective as an $A$-module is an injective $(A, H)$-Hopf module.
\end{itemize}
\end{corollary}

\begin{proof}
(1) Since any commutative algebra $A$ is a Poisson 3-Lie algebra with a trivial Poisson bracket, we have $A=A^A$, and any commutative $H$-comodule algebra $A$ is an $H$-comodule Poisson 3-Lie algebra. Any $H$-comodule is an $(\underline{A}, H)$-comodule with a trivial Lie $A$-action, and any $(A, H)$-Hopf module is a Poisson 3-Lie $(A, H)$-Hopf module with a trivial Lie $A$-action. So the map $\lambda$ in the proof of Theorem \ref{thm:2d} is 3-Lie $A$-linear. Also any homomorphism of $H$-comodules is a homomorphism of $(\underline{A}, H)$-comodules. The result follows from Theorem \ref{thm:2d} (1).

(2) Since any homomorphism of $(A, H)$-Hopf modules is a homomorphism of Poisson 3-Lie $(A, H)$-Hopf modules, the result follows from Theorem \ref{thm:2d}(2).
\end{proof}

\vspace{2mm}

\section{Fundamental theorem of Poisson 3-Lie Hopf modules}
\def\theequation{3.\arabic{equation}}
\setcounter{equation} {0}

In what follows, we will denote $B=A^{AcoH}$.
\begin{lemma}\label{lem:3a}
Let $M$ be a Poisson 3-Lie $B$-module. Then $A\ot_BM$ is a Poisson 3-Lie $(A,H)$-Hopf module, where the $A$-action and the 3-Lie $A$-action are given by
\begin{align*}
&x\cdot (a\ot_B m)=xa\ot_Bm\\
&(x,y)\di (a\ot_B m)=\{x,y, a\}\ot_Bm,
\end{align*}
and the $H$-coaction is given by
$$\rho_{A\ot_BM}(a\ot_B m)=a_{(0)}\ot_Bm\ot a_{(1)},$$
for all $a,x,y\in A,m\in M$.
\end{lemma}

\begin{proof}
It is obvious that the coaction is well defined and $A\ot_BM$ is an $(A,H)$-Hopf module. For all $a,x,y,z\in A,m\in M$. For $b\in B$,
\begin{align*}
(x,y)\di(ab\ot_Bm)&=(\{x,y,ab\})\ot_Bm\\
&\stackrel{(\ref{1e})}{=}(a\{x,y,b\}+\{x,y,a\}b)\ot_Bm\\
&=\{x,y,a\}b\ot_Bm\\
&=\{x,y,a\}\ot_Bb\cdot m\\
&=(x,y)\di(a\ot_Bb\cdot m),
\end{align*}
so the 3-Lie action is well defined. We have
\begin{align*}
(xy,z)\di(a\ot m)&=\{xy,z,a\}\ot m\\
&=[x\{y,z,a\}+\{x,z,a\}y]\ot m\\
&=x\{y,z,a\}\ot m+\{x,z,a\}y\ot m\\
&=x\cdot(\{y,z\}\di(a\ot m))+y\cdot((x,z)\di(a\ot m)),
\end{align*}
so the relation (\ref{1e}) is satisfied. 
\begin{align*}
(x,y)\di(z\cdot(a\ot m))&=(x,y)\di(za\ot m)\\
&=\{x,y,za\}\ot m\\
&=z\{x,y,a\}\ot m+\{x,y,z\}a\ot m\\
&=z\cdot((x,y)\di m)+\{x,y,z\}\cdot(a\ot m),
\end{align*}
so the relation (\ref{1f}) is satisfied. Thus $A\ot_BM$ is a Poisson 3-Lie $A$-module. Moreover 
\begin{align*}
\rho_{A\ot_BM}((x,y)\di (a\ot_B m))&=((x,y)\di (a\ot_B m))_{(0)}\ot((x,y)\di (a\ot_B m))_{(1)}\\
&=(\{x,y,a\}\ot m)_{(0)}\ot(\{x,y,a\}\ot m)_{(1)}\\
&=\{x,y,a\}_{(0)}\ot m\ot \{x,y,a\}_{(1)}\\
&=\{x_{(0)},y_{(0)},a_{(0)}\}\ot  m\ot x_{(1)}y_{(1)}a_{(1)}\\
&\stackrel{(\ref{2b})}{=}(x_{(0)},y_{(0)})\di(a_{(0)}\ot  m)\ot x_{(1)}y_{(1)}a_{(1)},
\end{align*}
the relation (\ref{2b}) is satisfied. The proof is completed.
\end{proof}

For every Poisson 3-Lie $(A,H)$-Hopf module $M$, from Lemma \ref{lem:2b} and Lemma \ref{lem:3a}, we deduce that $A\ot_BM^{AcoH}$ is a Poisson 3-Lie $(A,H)$-Hopf module.

\begin{lemma}\label{lem:3c}
Let $M$ be a Poisson 3-Lie $(A,H)$-Hopf module. Then the linear map $\alpha:A\ot_B M^{AcoH}\rightarrow M,\ a\ot_Bm\mapsto a\cdot m$ is a homomorphism of Poisson 3-Lie $(A,H)$-Hopf modules.
\end{lemma}

\begin{proof}
Clearly $\alpha$ is well defined and a left $A$-module map. For $x,y,a\in A,m\in M^{AcoH}$,
\begin{align*}
\alpha((x,y)\di(a\ot_Bm))&=\alpha(\{x,y,a\}\ot_Bm)\\
&=\{x,y,a\}\cdot m\\
&=(x,y)\di(a\cdot m)-z\cdot((x,y)\di m)\\
&=(x,y)\di(a\cdot m)\\
&=(x,y)\di\alpha(a\ot_Bm),
\end{align*}
where the fourth equality holds since $(x,y)\di m=0$. Hence $\alpha$ is a 3-Lie $A$-module map. It is straightforward that $\alpha$ is $H$-colinear.
\end{proof}

Consider $H$ as a right $H$-comodule via $\Delta_H$ and assume that there exists a homomorphism of right $H$-comodule $\phi:H\rightarrow A$ such that $\phi(1_H)=1_A$. For any Poisson 3-Lie $(A,H)$-Hopf module $M$, define linear map
$$p_M:M\rightarrow M,\quad m\mapsto\phi(S^{-1}(m_{(1)}))\cdot m_{(0)}.$$

\begin{lemma}\cite{Gu}
With the above notations, we have
\begin{itemize}
  \item [(1)] Im$p_M=M^{coH}$.
  \item [(2)] $p^2_M=p_M$, and $p_M$ is a projection.
\end{itemize}
\end{lemma}

Let $M\in\! _{\mathcal{P}A}\mathcal{M}^H$. For all $x,y\in A,m\in M$, put
$$(x,y)\di' m=p_M((x,y)\di m).$$

\begin{lemma}\label{lem:3d}
Let $H$ be a weak Hopf algebra. Let $M\in\! _{\mathcal{P}A}\mathcal{M}^H$. Suppose that there exists a right $H$-colinear algebra map $\phi:H\rightarrow A^A$, then for all $x,y,a\in A,m\in M$,
\begin{itemize}
  \item [(1)] $p_M(a\cdot m)=p_A(a)\cdot p_M(m),$
  \item [(2)] $(x,y)\di'p_M(m)=(x,y)\di' m,$
  \item [(3)] $(M,\di')$ is a 3-Lie $A$-module,
  \item [(4)] $(x,y)\di p_M(m)=\phi(x_{(1)}y_{(1)})\cdot ((x_{(0)},y_{(0)})\di'p_M(m))$,
  \item [(5)] $\phi(m_{(1)})\cdot p_M(m_{(0)})=m$.
\end{itemize}
\end{lemma}

\begin{proof}
We only verify the conditions (2),(3) and (4).

(2) For all $x,y\in A,m\in M$,
\begin{align*}
(x,y)\di'p_M(m)&=p_M((x,y)\di p_M(m))\\
&=p_M((x,y)\di (\phi(S^{-1}(m_{(1)}))\cdot m_{(0)}))\\
&=p_M(\{x,y,\phi(S^{-1}(m_{(1)}))\}\cdot m_{(0)}+\phi(S^{-1}(m_{(1)}))\cdot((x,y)\di m_{(0)}))\\
&=p_M(\phi(S^{-1}(m_{(1)}))\cdot((x,y)\di m_{(0)}))\\
&=p_A(\phi(S^{-1}(m_{(1)})))\cdot p_M((x,y)\di m_{(0)})\\
&=\phi(S^{-2}(m_{(1)2}))\phi(S^{-1}(m_{(1)3}))\cdot p_M((x,y)\di m_{(0)})\\
&=p_M((x,y)\di m)\\
&=(x,y)\di' m.
\end{align*}

(3) For all $x,y,u,v\in A,m\in M$, 
$$(x,y)\di' m=p_M((x,y)\di m)=-p_M((y,x)\di m)=-(y,x)\di' m.$$
Then
\begin{align*}
&(\{x,y,u\},v)\di' m+(u,\{x,y,v\})\di' m\\
&=(\{x,y,u\},v)\di' p_M(m)+(u,\{x,y,v\})\di' p_M(m)\\
&=p_M((\{x,y,u\},v)\di p_M(m))+p_M((u,\{x,y,v\})\di p_M(m))\\
&\stackrel{(\ref{1c})}{=}p_M((x, y)\di((u,v)\di p_M(m))-(u,v)\di((x, y)\di p_M(m)))\\
&=(x, y)\di'((u,v)\di p_M(m))-(u,v)\di'((x, y)\di p_M(m))\\
&=(x, y)\di' p_M((u,v)\di p_M(m))-(u,v)\di'p_M((x, y)\di p_M(m))\\
&=(x, y)\di' ((u,v)\di' p_M(m))-(u,v)\di'((x, y)\di' p_M(m))\\
&=(x, y)\di'((u,v)\di' m)-(u,v)\di'((x, y)\di' m),
\end{align*}
and
\begin{align*}
&(\{x,y,u\},v)\di' m=(\{x,y,u\},v)\di' p_M(m)=p_M((\{x,y,u\},v)\di p_M(m))\\
&\stackrel{(\ref{1d})}{=}p_M((y,u)\di((x,v)\di p_M(m)))-p_M((x,u)\di((y,v)\di p_M(m)))+p_M((x,y)\di((u,v)\di p_M(m)))\\
&=(y,u)\di'((x,v)\di p_M(m))-(x,u)\di'((y,v)\di p_M(m))+(x,y)\di'((u,v)\di p_M(m))\\
&=(y,u)\di'p_M((x,v)\di p_M(m))-(x,u)\di'p_M((y,v)\di p_M(m))+(x,y)\di'p_M((u,v)\di p_M(m))\\
&=(y,u)\di'((x,v)\di' p_M(m))-(x,u)\di'((y,v)\di' p_M(m))+(x,y)\di'((u,v)\di' p_M(m))\\
&=(y,u)\di'((x,v)\di' m)-(x,u)\di'((y,v)\di' m)+(x,y)\di'((u,v)\di' m).
\end{align*}
Therefore $(M,\di')$ is a 3-Lie $A$-module.

(4) Let $x,y\in A,m\in M$,
\begin{align*}
&\phi(x_{(1)}y_{(1)})\cdot ((x_{(0)},y_{(0)})\di'p_M(m))\\
&=\phi(x_{(1)}y_{(1)})\cdot ((x_{(0)},y_{(0)})\di'm)\\
&=\phi(x_{(1)}y_{(1)})\cdot p_M(((x_{(0)},y_{(0)})\di m))\\
&=\phi(x_{(1)}y_{(1)})\phi(S^{-1}(((x_{(0)},y_{(0)})\di m)_{(1)}))\cdot (x_{(0)},y_{(0)})\di m)_{(0)}\\
&=\phi(x_{(1)2}y_{(1)2})\phi(S^{-1}(x_{(1)1}y_{(1)1}m_{(1)}))\cdot ((x_{(0)},y_{(0)})\di m_{(0)})\\
&=\phi(S^{-1}(m_{(1)}))\cdot ((x,y)\di m_{(0)})\\
&=(x,y)\di(\phi(S^{-1}(m_{(1)}))\cdot m_{(0)})-\{x,y,\phi(S^{-1}(m_{(1)}))\}\cdot m_{(0)}\\
&=(x,y)\di p_M(m).
\end{align*}
The proof is completed.
\end{proof}

Since the $H$-comodule Poisson 3-Lie algebra $A$ itself is a Poisson 3-Lie $(A,H)$-Hopf module, $A$ is a 3-Lie $A$-module under the action given by
$$(x,y)\di'a=p_A(\{x,y,a\}), $$
for all $x,y,a\in A$. Because $\textup{Im}p_A=A^{coH}$, we have $A^{coH}$ is a 3-Lie $A$-submodule of $A$ under the action $\di'$. Actually for any Poisson 3-Lie $(A,H)$-Hopf module $M$, $(M^{coH},\di')$ is a 3-Lie $A$-submodule of $M$. However $M^{coH}$ in general is not a 3-Lie $A$-submodule under the action $\di$.

\begin{lemma}\label{lem:3e}
Let $\phi:H\rightarrow A^A$ be a right $H$-colinear algebra map, and $M\in\! _{\mathcal{P}A}\mathcal{M}^H$.
\begin{itemize}
  \item [(1)] If the 3-Lie action $\di'$ on $M^{coH}$ is trivial, then $M^{AcoH}=M^{coH}$.
  \item [(2)] If the 3-Lie action $\di'$ on $A^{coH}$ is trivial, then $A^{AcoH}=A^{coH}$.
\end{itemize}
\end{lemma}

\begin{proof}
Let $x,y\in A,m\in M^{coH}$, by Lemma \ref{lem:3d}(4), 
$$(x,y)\di m=(x,y)\di p_M(m)=\phi(x_{(1)}y_{(1)})\cdot ((x_{(0)},y_{(0)})\di'p_M(m))=0,$$
thus $m\in M^A$, that is, $m\in M^{AcoH}$.
\end{proof}

\begin{theorem}\label{thm:3f}
Let $M$ be a Poisson 3-Lie $(A, H)$-Hopf module and $\phi:H\rightarrow A^A$ a right $H$-colinear algebra map. Suppose $M^{coH}$ and $A^{coH}$ are trivial 3-Lie $A$-modules under $\di'$. Then the linear map $\alpha$ given in Lemma \ref{lem:3c} is an isomorphism of Poisson 3-Lie $(A, H)$-Hopf modules.
\end{theorem}

\begin{proof}
By Lemma \ref{lem:3c}, $\alpha$ is a homomorphism of Poisson 3-Lie $(A,H)$-Hopf modules. Let $a\in A,m\in M$, by Lemma \ref{lem:3e}, we have $p_M(m)\in M^{AcoH}$. So the following linear map is well defined:
$$\beta:M\rightarrow A\ot_B M^{AcoH},\ m\mapsto\phi(m_{(1)})\ot_B p_M(m_{(0)}).$$
It is straightforward to check that $\alpha\circ\beta=id_{M}$ and $\beta\circ\alpha=id_{A\ot_BM^{AcoH}}$. Therefore $\alpha$ is an isomorphism of Poisson 3-Lie $(A,H)$-Hopf modules.
\end{proof}

\begin{example}\label{ex:finite-dimensional-comodule-poisson3lie}
Let $H=k\{1,g,g^2|g^3=1\}$ the group algebra. 
Set $A=C\oplus M$ as a $k$‑vector space, where
\[
C=k\{c_0,c_1,c_2\},\qquad M=k\{x_0,x_1,x_2\},
\]
so $\dim A=6$. We define the commutative associative multiplication on $A$ as follows:
\begin{align*}
&1_A=c_0,\ c_ic_j=c_{i+j\pmod 3},\\
&x_ix_j=\begin{cases}
x_{i+j}, & i+j<3,\\
0, & i+j\ge 3;
\end{cases}\\
&C\cdot M=M\cdot C=0.
\end{align*}

The right $H$-coaction $\rho_A\colon A\to A\otimes H$ is defined by
\[
\rho_A(c_i)=c_i\otimes g^i,\quad \rho_A(x_i)=x_i\otimes g^i,\qquad i=0,1,2.
\]

One directly verifies that $\rho_A(ab)=\rho_A(a)\rho_A(b)$ for all $a,b\in A$, hence $A$ is a right $H$-comodule algebra with $A^{coH}=k\{1_A,x_0\}$.

Equip $A$ with the trivial Poisson $3$-Lie bracket
\[
\{a,b,c\}=0,\qquad \forall\,a,b,c\in A.
\]
Then $A$ is a right $H$-comodule Poisson $3$-Lie algebra with $A^A=A$.

Define the map $\phi\colon H\to A^A$ by $\phi(g^i)=c_i$ for $i=0,1,2$. Then it is easy to check that $\phi$ is a right $H$-colinear algebra homomorphism, and Theorem \ref{thm:3f} could be applied when viewing $A$ as a Poisson 3-Lie $(A,H)$-Hopf module.
\end{example}

Let $A$ be an $H$-comodule Poisson 3-Lie algebra. Let $I$ be a vector subspace of $A$. We say that $I$ is
\begin{itemize}
    \item[(i)] an $H$-ideal of $A$ if $I$ is an ideal of $A$ and an $H$-subcomodule of $A$;
    \item[(ii)] a Poisson 3-Lie ideal of $A$ if $I$ is an ideal of $A$ and a Lie ideal of $A$;
    \item[(iii)] a Poisson 3-Lie $H$-ideal of $A$ if $I$ is an $H$-ideal of $A$ and a Poisson 3-Lie ideal of $A$.
\end{itemize}

An $H$-comodule Poisson 3-Lie algebra $A$ is Poisson $H$-simple if the only Poisson 3-Lie $H$-ideals of $A$ are 0 and $A$.

\begin{lemma}\label{lem:3g}
Let $A$ be a Poisson $H$-simple $H$-comodule Poisson 3-Lie algebra. Then $A^{AcoH}$ is a field.
\end{lemma}

\begin{proof}
Let $b\in A^{AcoH}, b\neq 0$, then $Ab$ is an ideal of $A$ since $A$ is commutative. For all $a,x,y\in A$, we have
$$(x,y)\di(ab)=\{x,y,ab\}=\{x,y,b\}a+\{x,y,a\}b=\{x,y,a\}b\in Ab,$$
that is, $Ab$ is a 3-Lie ideal of $A$. It is clear that $Ab$ is an $H$-subcomodule of $A$. Therefore $Ab$ is a Poisson 3-Lie idea of $A$. Since $Ab\neq0$ and $A$ is Poisson $H$-simple, we obtain $Ab=A$, so there exists $a'\in A$ such that $a'b=1$, namely, $b$ is invertible.
\end{proof}

By Theorem \ref{thm:3f} and Lemma \ref{lem:3g}, we immediately obtain the following corollary.

\begin{corollary}
Let $A$ be a Poisson $H$-simple $H$-comodule Poisson 3-Lie algebra and $M$ a Poisson 3-Lie $(A, H)$-Hopf module. Assume that there exists a right $H$-colinear map $\phi:H\rightarrow A^A$ which is also an algebra map, and that $M^{AcoH}$ and $A^{AcoH}$ are trivial Lie $A$-module under the action of $\di'$. Then $M$ is a free as an $A$-module with rank equal to the dimension of $M^{AcoH}$ over $A^{AcoH}$.
\end{corollary}

Let $M,N\in\! _{\mathcal{P}A}\mathcal{M}^H$ and the morphism $f:M\rightarrow N$. For $m\in M^{AcoH},x,y\in A$, we compute
\begin{align*}
&f(m)_{(0)}\ot f(m)_{(1)}=f(m_{(0)})\ot m_{(1)}=f(m)\ot 1_H,\\
&(x,y)\di f(m)=f((x,y)\di m)=0,
\end{align*}
thus $f(m)\in N^{AcoH}$ and $f(M^{AcoH})\subseteq N^{AcoH}$. This gives rise to a functor 
$$G:\!  _{\mathcal{P}A}\mathcal{M}^H\rightarrow \! _B\mathcal{M},\quad M\mapsto M^{AcoH}.$$
By Lemma \ref{lem:3a}, we also get a functor 
$$F:\! _B\mathcal{M}\rightarrow\!  _{\mathcal{P}A}\mathcal{M}^H,\quad M\mapsto A\ot_BM.$$

\begin{proposition}\label{pro:3f}
Let $M\in\! _{\mathcal{P}A}\mathcal{M}^H$ and $N\in\! _B\mathcal{M}$. There is a functorial isomorphism of transposed Poisson $(A, H)$-Hopf modules
\begin{align*}
\psi:\! _{\mathcal{P}A}\textup{Hom}^H(A\ot_B&N,M)\rightarrow \textup{Hom}_B(N,M^{AcoH})\\
&f\mapsto\psi(f):N\rightarrow M^{AcoH},\ n\mapsto f(1_A\ot_B n).
\end{align*}
Thus the functors $F$ and $G$ form an adjoint pair with unit and counit
\begin{align*}
&\eta_N:N\rightarrow (A\ot_BN)^{AcoH},\ n\mapsto 1_A\ot_B n,\\
&\epsilon_M:A\ot_B M^{AcoH}\rightarrow M,\ a\ot_B m\mapsto a\cdot m.
\end{align*}
\end{proposition}

\begin{proof}
For all $x,y\in A,n\in N$,
$$(x,y)\di f(1_A\ot_Bn)=f((x,y)\di(1_A\ot_Bn))=f(\{x,y,1_A\}\ot_Bn)=0,$$
hence $f(1_A\ot_Bn)\in M^{A}$. Since $f$ is a morphism of $H$-comodule, we could obtain that $f(1_A\ot_Bn)\in M^{coH}$. Thus $f(1_A\ot_Bn)\in M^{AcoH}$. Also it is straightforward to verify that $\psi(f)$ is an $B$-module map.  Therefore $\psi$ is well defined. 

Now define
\begin{align*}
\psi':\textup{Hom}_B(N,&M^{AcoH})\rightarrow\! _{\mathcal{P}A}\textup{Hom}^H(A\ot_BN,M)\\
&g\mapsto\psi'(g):A\ot_B N\rightarrow M,\ a\ot_B n\mapsto a\cdot g(n).
\end{align*}
For any $a,a'\in A,n\in N$, we have
\begin{align*}
&\psi'(g)((x,y)\di(a\ot_B n))=\psi'(g)(\{x,y,a\}\ot_B n)\\
&=\{x,y,a\}\cdot g(n)=(x,y)\di(a\cdot g(n))-a\cdot((x,y)\di g(n))\\
&=(x,y)\di(a\cdot g(n)),
\end{align*}
where the fourth identity holds since $g(n)\in M^{AcoH}$, that is, $\psi'(g)$ is Lie $A$-linear. Easy to see that $\psi'(g)$ is $A$-linear and $H$-colinear. Thus $\psi'$ is also well defined. It is a routine exercise to check that $\psi\circ\psi'$ and $\psi'\circ\psi$ are respectively the identity of $\textup{Hom}_B(N,M^{AcoH})$ and $_{\mathcal{P}A}\textup{Hom}^H(A\ot_BN,M)$. The proof is completed.
\end{proof}

Finally we obtain the following corollary.
\begin{corollary}
Let $M$ be a Poisson 3-Lie$(A, H)$-Hopf module. Assume that there is a right $H$-colinear algebra map $\phi: H \rightarrow A^A$ and that $M^{coH}$ and $A^{coH}$ are trivial 3-Lie $A$-modules under $\di'$.
Then the functor $G=(-)^{AcoH}:\!  _{\mathcal{P}A}\mathcal{M}^H\rightarrow {}_{B}\mathcal{M}$ is dual Maschke, that is, every object of $ _{\mathcal{P}A}\mathcal{M}^H$ is $G$-relative projective.
\end{corollary}

\begin{proof}
The result follows from Theorem \ref{thm:3f}, Proposition \ref{pro:3f} and \cite[Theorem 3.4]{CM03}.
\end{proof}

\section*{Acknowledgement}

This work was supported by the NNSF of China (Nos. 12271292, 11901240).

\bibliographystyle{amsplain}

\end{document}